\documentclass[]{article}
\usepackage{amssymb,amsmath,amsthm,bm,setspace,mathrsfs}

\usepackage{url,color}%units}
\input cyracc.def
\newcommand{\nicefrac}{\frac}

\newcommand{\Z}{\mathbf{Z}}

\newcommand{\R}{\mathbf{R}}

\newcommand{\be}{\begin{equation}}
\newcommand{\ee}{\end{equation}}

\newcommand{\lip}{\text{\rm Lip}_\sigma }

\renewcommand{\P}{\mathrm{P}}
\newcommand{\E}{\mathrm{E}}

\newcommand{\1}{\boldsymbol{1}}
\renewcommand{\d}{{\rm d}}

\newcommand{\e}{{\rm e}}

\renewcommand{\ge}{\geqslant}
\renewcommand{\le}{\leqslant}

\author{Mathew Joseph\\University of Sheffield\and
Davar Khoshnevisan \\University of Utah \and Carl Mueller \\University of Rochester }
\title{Strong invariance and noise-comparison\\principles for some 
	parabolic stochastic PDEs\thanks{%
	Research supported in part by the United States
	National Science Foundation grants DMS-0747758 (M.J.), 
	DMS-1307470 (M.J. and D.K.), and DMS-1102646. (C.M.).}}
\date{April 28, 2014}
\newtheorem{stat}{Statement}[section]
\newtheorem{proposition}[stat]{Proposition}
\newtheorem{corollary}[stat]{Corollary}
\newtheorem{theorem}[stat]{Theorem}
\newtheorem{lemma}[stat]{Lemma}
\theoremstyle{definition}

\newtheorem*{OP}{Open Problem}
\newtheorem{example}[stat]{Example}
\newtheorem{assumption}[stat]{Assumption}

 \numberwithin{equation}{section} %
\begin{document}%\onehalfspacing
\maketitle
\begin{abstract}%\\
We consider a system of interacting diffusions on the integer lattice. By 
letting the mesh size go to zero and by using a suitable scaling, we show that 
the system converges (in a strong sense) to a solution of the stochastic heat 
equation on the real line. As a consequence, we obtain comparison inequalities for 
product moments of the stochastic heat equation with different nonlinearities.

\end{abstract}

\section{Introduction} 
We propose to study the solution 
$u:=\{u_t(x)\}_{t\ge 0,x\in\R}$ to the following partial differential 
equation:\footnote{Note that $u_t(x)$ denotes the evaluation of the solution 
at $(t\,,x)$; this is consistent with the standard nomenclature of stochastic 
process theory. To be consistent, we will \underline{never} write $u_t$ in 
place of the time derivative $\partial_t u$ of $u$.}
\[
	\frac{\partial}{\partial t} u_t (x) = (\mathcal{L} u_t)(x) + \sigma(u_t(x))\xi(t\,,x)
	\qquad(t>0,\,x\in\R);
	\tag{SHE}
\]
subject to $u_0(x):=1$ for all $x\in\R$. 
The function $\sigma:\R\to\R$ is assumed to be non-random
and globally Lipschitz continuous. The
forcing term $\xi$ in (SHE)
is \emph{space-time white noise}. That is, $\xi$ is assumed to be
a generalized Gaussian random field
on $\R_+\times\R$ with mean function zero and covariance measure
\begin{equation}
	\textnormal{Cov}(\xi(t\,,x)~,~\xi(s\,,y)) =\delta_0(t-s)\delta_0(x-y).
\end{equation}
Finally, the operator $\mathcal{L}$
acts on the variable $x$ only, and is the fractional Laplacian, on $\R$, of order 
$\alpha$. That is,
\begin{equation}
	\mathcal{L} = -\nu(-\Delta)^{\alpha/2},
\end{equation}
where $\nu>0$ is a parameter that is held fixed throughout this paper
[$\nu$ is the so-called ``viscosity coefficient''].

We are interested only in the cases where the semigroup 
$t\mapsto\exp(-t\mathcal{L})$ is positive and conservative; that is precisely 
when the fundamental solution to the heat operator 
$(\partial/\partial t) -\mathcal{L}$ is a probability density
at all times.  It is 
well known that this property holds if and only if $\alpha\in(0\,,2]$.  
Moreover, in that case, $\mathcal{L}$ is a negative-definite 
pseudo-differential operator with Fourier symbol 
$\hat{\mathcal{L}}(z) = -\nu|z|^\alpha$, where our Fourier transform is 
normalized so that
\begin{equation}
	\hat{f}(z) = \int_{-\infty}^\infty \e^{izx}f(x)\,\d x
	\qquad\text{for all $z\in\R$ and $f\in L^1(\R)$}.
\end{equation}
In probabilistic terms, the preceding means that we can 
identify $\mathcal{L}$ as the generator of a symmetric 
$\alpha$-stable L\'evy process $\{S_t\}_{t\ge 0}$ whose 
characteristic function at time $t\ge 0$ is normalized as
\begin{equation}\label{eq:S}
	\E\exp(izS_t)=
	\exp(-\nu t\vert z\vert^{\alpha})\qquad(t\ge 0,\, z\in\R).
\end{equation}

Even though $\alpha\in(0\,,2]$ always leads us 
to a nice pseudo-differential operator $\mathcal{L}$, we will 
further assume that
\begin{equation}\label{eq:alpha}
	1<\alpha\le 2. 
\end{equation}
It is well known that under the condition \eqref{eq:alpha}, 
(SHE) has an almost surely unique H\"older-continuous solution $u$ that is 
predictable with respect to the 
filtration of the white noise $\xi$ and satisfies
\begin{equation}
	\sup_{t\in[0,T]}\sup_{x\in\R}\E\left(|u_t(x)|^k\right)<\infty
	\qquad\text{for all $T>0$ and $k\ge 2$}.
\end{equation}
See, for example, Dalang \cite{Dalang:99} and Peszat and
Zabczyk \cite{PZ},
and also Foondun and Khoshnevisan \cite{FK} for proofs of these
assertions.
Moreover, if $\alpha\in(0\,,1]$ and $\sigma(u)=1$, then 
(SHE) does not have a function solution \cite{Dalang:99,PZ}.

We may think of $u_t(x)$ as the conditional expected density of a branching
particle system at space-time point $(t\,,x)$ given the white noise $\xi$, where the particles
move in $\R$ as an $\alpha$-stable L\'evy process. Moreover,
these particles interact with one another only through the
random environment $\xi$ and the function $\sigma$; the latter describes the
underlying [non-linear] branching mechanism.

Let $\bar\sigma:\R\to\R$ be a non random Lipschitz continuous function
and suppose $\bar{u}$ solves (SHE) but with $\sigma$ replaced by $\bar\sigma$; that is,
\begin{equation}
	\frac{\partial}{\partial t} \bar{u}_t (x) = (\mathcal{L} 
	\bar{u}_t)(x) + \bar\sigma(\bar{u}_t(x))\xi(t\,,x)
	\qquad(t>0,\,x\in\R),
\end{equation}
subject to the same initial condition as $u$ was; namely,
$\bar{u}_0(x):= 1$ for all $x\in\R$.
One of the goals of this project is to prove that if
$\sigma\le\bar\sigma$ pointwise, then frequently,
\begin{equation}\label{comp}
	\E\left[\prod_{i=1}^m u_t(x_i)\right] \le \E\left[\prod_{i=1}^m \bar{u}_t(x_i)\right], 
\end{equation}
for all $t>0$ and $x_1,\ldots,x_m\in\R$. We will conclude this
fact---see Theorem \ref{th:comp}---from a general approximation result that we describe next.

Let $\mathscr{L}$ denote the generator of a rate $1$, continuous-time random walk on $\Z$.  That is, there is a probability measure $q$ on $\Z$ such that 
\begin{equation}
	(\mathscr{L}f) (m) := \sum_n q(n-m) \big[f(n)-f(m)\big],
\end{equation}
for appropriate functions $f:\Z\to\R$. 
Consider the interacting particle system 
\[
	\d U_t(x) = (\mathscr{L} U_t) (x)\, \d t
	 + \sigma ( U_t (x)  ) \,\d B_t(x) \quad(t>0\,,x\in\Z),
	\tag{SHE$_1$}
\]
subject to $U_0(x)\equiv 1$ for all $x\in\Z$. Here, the
jump rate of the underlying walk is assumed to be one, and $\tilde{\xi}(t\,,x):=
\d B_t(x)/\d t$ denotes space-time white noise on $\R_+\times\Z$;
equivalently, $\{B_\bullet(x)\}_{x\in\Z}$ is a countable family of independent
standard Brownian motions on $\R$, and (SHE$_1$) is understood as an 
infinite system of It\^o stochastic differential equations in the
sense of Shiga and Shimizu \cite{ShigaShimizu}
[see \S\ref{subsec:Mild:DSHE} below for more details]. 

One might expect that \textit{under appropriate conditions on the generator $\mathscr{L}$}, if we first rescale space in (SHE$_1$) so that
$x\in\varepsilon\Z$ instead of $x\in\Z$, and then speed up the
rate of the jumps of the underlying walks suitably, then the resulting rescaled
version of (SHE$_1$) converges to the continuum equation (SHE) as 
$\varepsilon\downarrow 0$. This is an example of [an infinite dimensional]
invariance principle.

This program has been carried out successfully
by Funaki \cite{Funaki}
in the special case that $\mathscr{L}$ denotes the generator of a simple random
walk, where $\nu=1$ and $\R$ is replaced by $[0\,,1]$ in (SHE).%\textcolor{red}{We also mention here the impressive body of literature on the numerical analysis of SPDEs, particularly by Gy\"{o}ngy and Krylov (see \cite{GK1} and the references therein). The paper \cite{GK1} deals with the approximation of  }

Our main contribution includes an extension of Funaki's program 
in two  directions:
Firstly, we study the full equation on $\R$, in place of $[0\,,1]$,
where compactness problems arise in noteworthy ways. 
Secondly, we establish an invariance principle of a sort that we will 
explain next.

In order to describe our invariance result, let $\pi_\varepsilon$ denote the rescaling
that maps the parameter set $\varepsilon\Z$ to the rescaled parameter set
$\Z$. That is,
\begin{equation}\label{pi:eps}
	(\pi_\varepsilon f)(x) := f(\varepsilon x)\qquad(x\in\Z\,,\varepsilon\in(0\,,1]),
\end{equation}
for all functions $f:\varepsilon\Z\to\R$. Now we consider the semi-discrete
stochastic heat equation (SHE$_1$), rescaled so that $x\in\varepsilon\Z$.
That is, we study the solution $U^{(\varepsilon)}$ to the following infinite system
of stochastic differential equations: For $t>0$ and $x\in\varepsilon\Z$,
\[
	\d U_t^{(\varepsilon)}(x) = \varrho_\varepsilon\times
	\left(\mathscr{L} \left[\pi_\varepsilon U_t^{(\varepsilon)} \right] \right)
	 (x/\varepsilon)\, \d t + \frac{1}{\sqrt\varepsilon}\,\sigma 
	 \left( U_t^{(\varepsilon)}(x) \right)
	\d B_t(x/\varepsilon),
	\tag{SHE$_\varepsilon$}
\]
subject to $U_0^{(\varepsilon)}(x) \equiv 1$ for all $x\in\varepsilon\Z$,
where $\varrho_\varepsilon>0$ is a suitable normalization.

As it turns out, the following is a good choice for our normalization factor:
\begin{equation}\label{varrho}
	\varrho_\varepsilon:=\varepsilon^{-\alpha}
	\qquad(0<\varepsilon\le 1).
\end{equation}
[In fact, this factor is, in a sense, uniquely optimal.]
Our choice for $\varrho_\varepsilon$ is justified by the following corollary to our main theorem Theorem \ref{th:main}, stated in Section \ref{sec:main}.

\begin{corollary}[A weak invariance principle]\label{co:invariance}
	Suppose $\mathscr{L}$ is the generator of a
	symmetric random walk
	$\{X_t\}_{t\ge 0}$ on $\Z$ whose characteristic function satisfies 
	Assumption \ref{ass} below. Then,
	the finite-dimensional distributions of $U^{(\varepsilon)}$---normalized
	via \eqref{varrho}---converge to those
	of $u$ as $\varepsilon\downarrow 0$.
\end{corollary}

In fact, there is a very strong sense in which ``functional
weak convergence'' takes place; see the strong 
approximation theorem--in the spirit of Cs\"org\H{o} and 
R\'ev\'esz \cite{CsorgoRevesz}---of Theorem \ref{th:SA} below.

In due time, we will show also that Corollary \ref{co:invariance} 
has a number of consequences
in the analysis of the solution to the [continuous] stochastic partial
differential equation (SHE);
see, for example: (i) Theorem \ref{th:SA}, which is a strong approximation
theorem; (ii) Theorem \ref{th:comp}, which is a comparison principle
for the diffusion coefficient $\sigma$ of (SHE); and (iii) Theorem \ref{th:lyapunov},
which establishes  nearly-optimal bounds on the moment Lyapunov
exponents of the solution to (SHE).

Our methods are connected loosely to earlier ideas from the literature 
on the numerical analysis of SPDEs.
See, in particular, Gy\"{o}ngy and Krylov \cite{GK1,GK2,GK3}.

\section{Preliminaries, and main results}
In this section we review some basic facts about SPDEs, mainly in order
to state the notation, and then list the main contributions of the paper.

\subsection{White noise and integration}
It is well-known, and can be verified directly, that our space-time white noise
can be realized as the distributional mixed derivative,
\begin{equation}
	\xi(t\,,x) =  \frac{\partial^2 W(t\,,x)}{\partial t\partial x},
\end{equation}
of a \emph{two-sided Brownian sheet} $W:=\{W(t\,,x)\}_{t\ge 0,x\in\R}$. 
The latter object
is a mean-zero Gaussian process whose covariance function is 
\begin{equation}
	\textnormal{Cov} ( W(t\,,x) \,, W(s\,,y) )
	=\min(s\,,t)\times\min(|x|\,,|y|)\times \1_{(0,\infty)}(xy),
\end{equation}
for all $s,t\ge 0$ and $x,y\in\R$.

Standard results \cite[Chapter 1]{Walsh} imply that $W$ has 
H\"older-continuous sample functions.
Since $W_t:= W(t\,,\bullet)$ is an infinite-dimensional Brownian motion,
[with values say in $C(\R\,,\R)$],
we may then define the Walsh integral process
\begin{equation}
	M_t:=\int_{(0,t)\times \R}Z_s(y)\, W(\d s\,\d y)
	\qquad(t\ge 0),
\end{equation}
in the spirit of an It\^o integral,
as a continuous martingale with
quadratic variation 
\begin{equation} \label{eq:qv}
	\langle M\rangle_t =\int_0^t\d s\int_{-\infty}^\infty\d y\
	[Z_s(y)]^2 \qquad(t\ge 0),
\end{equation}
provided that $Z:=\{Z_t(\bullet)\}_{t\ge 0}$ is
predictable with respect to the Brownian filtration generated by
$\{W_t\}_{t\ge 0}$, and the preceding displayed integral is
finite for all $t\ge 0$.
See Walsh \cite[Chapter 2]{Walsh}. We mention the following variation of
the Burkholder--Davis--Gundy inequality (Theorem 1.1 in \cite{BD}, also see Remark 2.2 in \cite{FK}):  For every
$k\ge 2$ and $t\ge 0$,
\begin{align}
	\|M_t\|_k^2
	\le (4k)^{1/k}\int_0^t \d s\int_{-\infty}^\infty\d y\
	\|Z_s(y)\|_k^2,
	\tag{\textnormal{BDG}}
\end{align}
where, here and throughout, we adopt the standard notation,
\begin{equation}
	\|X\|_\nu := \left\{ \E\left(|X|^\nu\right)\right\}^{1/\nu},
\end{equation}
valid for every $\nu\in[1\,,\infty)$ and $X\in L^\nu(\Omega)$.

\subsection{The mild form of (SHE)}
Define $p_t(x)$ to be the transition density of a symmetric $\alpha$-stable 
L\'evy process. In light of \eqref{eq:S}, this means that
\begin{equation}
	\int_{-\infty}^\infty\e^{ixz}p_t(x)\,\d x=\exp(-\nu t|z|^\alpha)
	\qquad(t>0,z\in\R).
\end{equation}

Because 
$(s\,,t\,,x\,,y)\mapsto p_{t-s}(y-x)$ is the fundamental solution to the heat operator  
$(\partial/\partial t) - \mathcal{L}$, the theory of Walsh 
\cite[Chapter 3]{Walsh} tells us that
we may write (SHE) in mild form as follows:
\begin{equation}\label{mild}
	u_t(x) = 1 + \int_{(0,t)\times\R} p_{t-s}(y-x)\sigma(u_s(y))\,W(\d s\,\d y),
\end{equation}
for all $t>0$ and $x\in\R$. The integral is understood to be a stochastic It\^o-type
integral in the sense of Walsh \cite[Chapter 2]{Walsh}.

\subsection{The mild form of (SHE$_\varepsilon$)}\label{subsec:Mild:DSHE}
Before we write (SHE$_\varepsilon$) in mild form, it might help to recall a few facts about
continuous-time random walks. We use this opportunity to also set forth some
notation. 

Recall that $X:=\{X_t\}_{t\ge 0}$ denotes a rate-one symmetric
continuous-time random walk on $\Z$, whose generator is $\mathscr{L}$ 
and satisfies $X_0=0$.  [We further assume that the random walk satisfies 
the technical Assumption \ref{ass} below. But that assumption is not
germane to the present discussion, and so will be introduced later on.]

Let us define, for every fixed $\varepsilon>0$,
a  random walk $X^{(\varepsilon)}:=\{X^{(\varepsilon)}_t\}_{t\ge 0}$ on 
the rescaled lattice $\varepsilon\Z$ by setting
\begin{equation}\label{eq:X:eps}
	X^{(\varepsilon)}_t := \varepsilon X_{t/\varepsilon^{\alpha}}
	\qquad(t\ge 0).
\end{equation}
We emphasize that $X^{(\varepsilon)}_0=0$ and
$\E X^{(\varepsilon)}_t = 0$ for all $t\ge 0$.
Thus, the random walk $X$ is rescaled to $X^{(\varepsilon)}$, so that
$X^{(\varepsilon)}$ might possibly converge weakly to an $\alpha$-stable 
L\'evy process $S$---see \eqref{eq:S}---as $\varepsilon\downarrow 0$.

The transition probabilities of each random walk $X^{(\varepsilon)}$ are 
\begin{equation}\label{P:eps}
	P^{(\varepsilon)}_t(x) := \P_0\{ X^{(\varepsilon)}_t = x\}
	\qquad(t\ge 0,\, x\in\varepsilon\Z).
\end{equation}

Let $\mathscr{L}^{(\varepsilon)}$ denote the generator of $X^{(\varepsilon)}$,
and choose and fix a bounded measurable function $\phi:\varepsilon\Z\to\R$.
Because 
\begin{equation}
	(\mathscr{L}^{(\varepsilon)}\phi)(x) =\varepsilon^{-\alpha}
	\left(\mathscr{L}(\pi_\varepsilon\phi) \right)(x/\varepsilon)
	\qquad\text{for all $x\in\varepsilon\Z$,}
\end{equation}
we may rewrite (SHE$_\varepsilon$) as 
the following system of It\^o  SDEs: For all $x\in\varepsilon\Z$,
\begin{equation}
	\d U_t^{(\varepsilon)}(x) =\varepsilon^{\alpha}\varrho_\varepsilon\cdot
	 \left( \mathscr{L}^{(\varepsilon)} U_t^{(\varepsilon)} \right)(x) \,\d t
	 + \varepsilon^{-1/2}\,\sigma\left( U_t^{(\varepsilon)}(x) \right)
	\d B_t(x/\varepsilon),
\end{equation}
subject to $U_0^{(\varepsilon)} \equiv 1$. From now on, we adopt the normalization
\eqref{varrho} without further mention. In that case, we have
\begin{equation}\label{eq:SHE'}
	\d U_t^{(\varepsilon)}(x) =
	 \left( \mathscr{L}^{(\varepsilon)} U_t^{(\varepsilon)} \right)(x) \,\d t
	 + \varepsilon^{-1/2}\,\sigma\left( U_t^{(\varepsilon)}(x) \right)
	\d B_t(x/\varepsilon) ,
\end{equation}
for $t>0$ and $x\in\varepsilon\Z$,
subject to $U_0^{(\varepsilon)} \equiv 1$.
It follows from elementary Markov chain theory that $(s\,,t\,,x\,,y)\mapsto
P^{(\varepsilon)}_{t-s}(y-x)$ is the fundamental solution to
the Kolmogorov--Fokker--Planck equation
$(\partial/\partial t)p - \mathscr{L}^{(\varepsilon)}p=0$. 
Therefore, we may rewrite \eqref{eq:SHE'}---equivalently (SHE$_\varepsilon$)---in mild
form as follows:
\begin{equation}\label{mild_d}
	U_t^{(\varepsilon)}(x) = 1 + \varepsilon^{-1/2}
	\sum_{y\in\varepsilon\Z}\int_0^t P_{t-s}^{(\varepsilon)}(y-x)
	\sigma\left( U_s^{(\varepsilon)}(y) \right)\d B_s(y/\varepsilon),
\end{equation}
valid for all $t>0$ and $x\in\varepsilon\Z$. According to Shiga
and Shimizu \cite{ShigaShimizu},
the system \eqref{mild_d} of interacting SDEs has a unique strong solution
$U^{(\varepsilon)}$ that is continuous in $t$ almost surely. We appeal to this fact
tacitly from now on.

\subsection{The main result, and some consequences}\label{sec:main}

We will first state the assumption needed for our main result. 
Let $\mu$ denote the \emph{dislocation distribution} of the random walk $X$ on $\Z$; 
that is, for all $A\subseteq\R$,
\begin{equation}\label{eq:mu:gamma}
	\mu(A):=\P\left\{X_\gamma\in A\right\}, 
	\quad \mbox{where} \quad \gamma :=\inf\{s>0:\, X_s\ne 0\}.
\end{equation}
The Fourier transform $\hat{\mu}$ of $\mu$ is computed easily as
\begin{equation}\label{eq:muhat}
	\hat{\mu}(z):=\sum_{x\in\Z} \e^{iz x} \mu(\{x\}) 
	=\E\exp(izX_1)
	\qquad(-\pi\le z\le \pi).
\end{equation}
We will refer to $\hat\mu$ as the \emph{characteristic function} of
the random walk $X$. It is possible to show that, if
\begin{equation}
	1 - \hat{\mu}(z) = \nu
	|z|^\alpha + o\left(|z|^\alpha\right)\qquad\text{as $|z|\to 0$},
\end{equation}
then the finite-dimensional distributions of the
normalized random walk $X^{(\varepsilon)}$---see \eqref{eq:X:eps}---converge
to those of a symmetric $\alpha$-stable
L\'evy process with transition function $p_t(x)$,
as $\varepsilon\downarrow 0$. We will need to strengthen
the preceding property for our results.
The following assumption restricts a little further the behavior of the characteristic function
of $\mu$ near $0$.
\begin{assumption} \label{ass}
	We assume that the dislocation measure $\mu$ satisfies the following:
	(i) $\{z\in[-\pi\,,\pi]:\ \hat\mu(z)=1\}=\{0\}$; and (ii)
	 There exists $a>0$ such that 
	 \begin{equation}
	 	1-\hat{\mu}(z) = \nu|z|^\alpha + O\left(|z|^{a+\alpha}\right)
		\qquad\text{as $|z|\to 0$}.
	 \end{equation}
\end{assumption}

\begin{example}[$\alpha=2$]\label{ex:alpha=2}
	Recall \cite{Spitzer} that the random walk $X$ is
	called \emph{strongly aperiodic} when
	$|\hat\mu(z)|=1$ if and only if $z$ is an integer multiple of
	$2\pi$. Thus, it follows
	that Assumption \ref{ass}(i) holds for all strongly aperiodic
	random walks. There are also random walks that satisfy Assumption
	\ref{ass}(i) but are not strongly aperiodic. The simple symmetric
	random walk [on $\Z$] is a good example. In fact,
	the simple walk satisfies all of Assumption \ref{ass}, with $\nu=\nicefrac12$. 
	More generally, we now show that if $Y_1=X_{\gamma}$ satisfies $\E(Y_1^2)=2\nu$ and
	$Y_1\in L^{2+\eta}(\Omega)$ for some $\eta>0$, then
	Assumption \ref{ass}(ii) holds for the same value $\nu=\frac12\E(Y_1^2)$. 
	[It might help to recall that the stopping time $\gamma$ was defined earlier
	in \eqref{eq:mu:gamma}.]
	Without loss of generality,
	we may---and will---assume that $\eta\in(0\,,1)$. 
	Equivalently, we are assuming that
	\begin{equation}\label{eq:3-delta}
		\E(|Y_1|^{3-\delta})<\infty\qquad
		\text{for some $\delta\in(0\,,1)$.}
	\end{equation}
	Because $|1-\cos x -(x^2/2)|\le C\min (x^2,|x|^3 )$
	for all $x\in\R$, and since $\E(Y_1^2)=2\nu$, 
	we plug in $x:=zY_1$ in the above inequality and integrate $[\d\P]$ in order to see that
	$| 1 - \hat\mu(z) - \nu z^2| \le Cz^2\E[
	\min( |z||Y_1|^3\,, Y_1^2)],$
	for all $z\in\R$. In particular, we can see that Assumption \ref{ass}
	holds as long as there exists $a>0$ such that
	\begin{equation}\label{ass:p}
		\E\left[ \min\left( |z||Y_1|^3\,, Y_1^2\right)\right] = O\left( |z|^a\right)
		\qquad\text{as $|z|\to 0$}.
	\end{equation}
	Next, let us write, for any value of $\delta\in(0\,,1)$,
	\begin{align}\notag
		\E\left[ \min\left( |z||Y_1|^3\,, Y_1^2\right)\right] &=
			|z|\E\left[|Y_1|^3;\, |Y_1|\le |z|^{-1}\right] + \E\left[
			Y_1^2;\, |Y_1|> |z|^{-1}\right]\\
		&\le |z|^{1-\delta}\E\left(|Y_1|^{3-\delta}\right)+ \E\left[
			Y_1^2;\, |Y_1|> |z|^{-1}\right].
	\end{align}
	We first apply H\"older's inequality and then Chebyshev's inequality
	in order to see that
	\begin{align}\notag
		\E\left[Y_1^2;\, |Y_1|> |z|^{-1}\right] &\le \left[\E\left(|Y_1|^{3-\delta}
			\right)\right]^{2/(3-\delta)}
			\left[\P\{|Y_1|>|z|^{-1}\}\right]^{(1-\delta)/(3-\delta)}\\
		&\le |z|^{1-\delta}\E\left(|Y_1|^{3-\delta}\right).
	\end{align}
	Thus we deduce \eqref{ass:p} from \eqref{eq:3-delta}.
	\qed
\end{example}

\begin{example}[$\alpha<2$]\label{ex:alpha<2}
	There are also random walks on $\Z$ that satisfy Assumption
	\ref{ass} when $1<\alpha<2$. Let us choose and fix $\alpha\in(1\,,2)$
	and define
	\begin{equation}
		\mu(\{j\}) := \frac{1}{2\zeta(\alpha+1)|j|^{\alpha+1}}
		\qquad\text{for $j=\pm1,\pm2,\ldots,$}
	\end{equation}
	where $\zeta$ denotes the zeta function. Evidently,
	$\mu$ defines a symmetric probability measure on $\Z$,
	and
	\begin{equation}
		1-\hat\mu(z) = \frac{1}{\zeta(\alpha+1)}\sum_{j=1}^\infty
		\frac{1-\cos(jz)}{j^{\alpha+1}}\qquad\text{for every $z\ge 0$}.
	\end{equation}
	Assumption \ref{ass}(i) follows readily from this representation. Also,
	a few lines of direct computation show that there exists a finite
	constant $C$ such that
	\begin{equation}
		\sum_{j=1}^\infty\left| \int_j^{j+1}\frac{1-\cos(xz)}{x^{\alpha+1}}\,\d x - 
		\frac{1-\cos(jz)}{j^{\alpha+1}}\right| \le C\max\left(
		z^2\,,z^{1+\alpha}\right),
	\end{equation}
	for all $z\ge 0$. This yields immediately that
	$1-\hat\mu(z) = \nu z^\alpha + O(z^2)$
	as $z\downarrow0$, where
	\begin{equation}
		\nu := \frac{1}{\zeta(\alpha+1)}\int_0^\infty
		\frac{1-\cos x}{x^{\alpha+1}}\,\d x=
		-\frac{\pi}{2\zeta(\alpha+1)\Gamma(\alpha+1)\sin(\alpha\pi/2)};
	\end{equation}
	see \cite[\S4, p.\ 13]{Kh:LCSN} for the details of this last computation.
	Thus, Assumption \ref{ass}(ii) also holds for the preceding
	value of $\nu$ [with $a:=2-\alpha>0$].\qed
\end{example}

The following is our main invariance principle.

\begin{theorem}\label{th:main}
	Suppose $\mathscr{L}$ is the generator of a symmetric random walk
	$\{X_t\}_{t\ge 0}$ on $\Z$ whose  characteristic 
	function satisfies Assumption \ref{ass}. 
	Then there exists a coupling of $\{U^{(\varepsilon)}\}_{\varepsilon\in(0,1)}$ 
	and $u$---all on  a common underlying probability space---such that
	\begin{equation}
		\sup_{t\in[0,T]}\sup_{x\in\varepsilon\Z}\E\left(\left|
		U^{(\varepsilon)}_t(x)-u_t(x) \right|^k\right) = 
		o(\varepsilon^{\rho k/2}) \qquad
		(\varepsilon\downarrow 0),
	\end{equation}
	for every fixed $T>0$, $k\ge 2$, and $\rho\in(0\,, \alpha-1)$.
\end{theorem}

We will see also that Theorem \ref{th:main} implies the following,
which in turns implies Corollary \ref{co:invariance} readily
[i.e., without need for  proof]. 

\begin{theorem}[Strong approximations]\label{th:SA} 
	Under the assumptions of Theorem \ref{th:main}, 
	the preceding coupling satisfies the following: For all $M,T>0$
	and $\rho\in(0\,,\alpha-1)$,
	\begin{equation}
		\varepsilon^{-\rho/2}\sup_{t\in(0,T)}\sup_{\substack{x\in\varepsilon\Z:\\
		|x|\le \varepsilon^{-M}}}\left| U^{(\varepsilon)}_t(x)
		-u_t(x)\right| \stackrel{\P}{\longrightarrow}0
		\qquad\text{as }\varepsilon\downarrow 0.
	\end{equation}
\end{theorem}

Using a comparison of moments argument for interacting diffusions, one can 
also obtain the following consequence of Theorem \ref{th:main}; this was alluded
to earlier.

\begin{theorem}[Comparison of moments] \label{th:comp} Let the assumptions of Theorem \ref{th:main} hold. Suppose $\bar{u}$ is the solution to \textnormal{(SHE)}
	subject to $\bar{u}_0\equiv 1$,
	where $\sigma$ is replaced everywhere by $\bar{\sigma}$
	for a function $\bar{\sigma}:\R\to\R$ that is Lipschitz continuous,
	and  $0\le \sigma(z)\le\bar{\sigma}(z)$ for all $z\in\R$. 
	Let us assume further that $\sigma(0)=\bar\sigma(0)=0$.
	Then,
	\begin{equation}
		\E \left[ \prod_{j=1}^m u_t(x_j) \right]
		\le \E \left[ \prod_{j=1}^m \bar{u}_t(x_j)\right],
	\end{equation}
	for all $t\ge 0$ and $x_1,\ldots,x_m\in\R$.
\end{theorem}

The following theorem gives upper and lower bounds on the exponential rate of growth of the moments (Lyapunov-exponents) of the stochastic heat equation when the nonlinearity is ``roughly" linear.

\begin{theorem}[Lyapunov-exponent bounds] \label{th:lyapunov}
	Suppose $u$ solves \textnormal{(SHE)} with 
	$\mathcal{L}f=\nu f''$ subject to initial profile
	$u_0\equiv 1$.
	If there exists nonnegative constants 
	$0\le {\rm L}_{\sigma}\le \lip<\infty$ such that ${\rm L}_\sigma 
	\vert z\vert \le  \sigma(z) \le \lip \vert z\vert $ for all $z$ , then 
	for all integers $k\ge 2$ and reals $x\in\R$,
	\begin{equation}\begin{split}
		&\frac{{\rm L}_\sigma^4 k(k^2-1)}{48\nu}
			\le \liminf_{t\to\infty} \frac1t \log\E\left(|u_t(x)|^k\right) \\
		&\hskip1.7in 
			\le \limsup_{t\to\infty} \frac1t \log\E\left(|u_t(x)|^k\right) 
			\le \frac{{\rm Lip}_\sigma^4 k^3}{\nu}.
	\end{split}\end{equation}
\end{theorem}

Ivan Corwin and Roger Tribe have independently asked us whether one
can transfer the knowledge of Lyapunov exponents for linear SPDEs
to nonlinear ones [private communications]. The preceding answers their
question in the affirmative.

Let us conclude this discussion by stating two open problems.

\begin{OP}
	Consider the parabolic Anderson model; that is, 
	(SHE) with $\sigma(x)=\lambda x$, where $\lambda\in\R$ is a fixed constant.
	According to a predication of Kardar \cite{Kardar}, one would expect that
	\begin{equation}\label{eq:Kardar}
		\lim_{t\to\infty}\frac1t \log\E\left(|u_t(x)|^k\right)
		=\frac{\lambda^4k(k^2-1)}{48\nu}\qquad\text{for every $x\in\R$}.
	\end{equation}
	[The left-hand side is independent of $x$, actually; see the proof of Lemma 18
	in Dalang \cite{Dalang:99}.] Bertini and Cancrini \cite{BC} have invented
	a rigorous Feynman--Kac moment formula, which suggests the Kardar formula
	\eqref{eq:Kardar} though their claim to proof of \eqref{eq:Kardar} is false
	because of an incorrect appeal to the Skorohod lemma. As far as we know
	a rigorous proof of \eqref{eq:Kardar} is still not available. 
	Our Theorem \ref{th:comp} shows that when ${\rm L}_\sigma \vert z\vert \le \sigma(z)\le {\rm Lip}_{\sigma} \vert z\vert$ for all $z$,
	\begin{equation}\label{AAB}
		\frac{({\rm L}_\sigma^4+o(1))k(k^2-1)}{48\nu}
		\le \frac1t\log \E\left(|u_t(x)|^k\right) \le \frac{%
		({\rm Lip}_\sigma^4+o(1))k^3}{\nu},
	\end{equation}
	as $t\to\infty$.
	In the case that the initial profile is a point mass [instead of $u_0(x)\equiv1$,
	as is the case here],
	Borodin and Corwin \cite[\S4.2]{BorodinCorwin} have succeeded in 
	proving \eqref{eq:Kardar}. Can their methods be used to establish
	\eqref{eq:Kardar} when $u_0(x)\equiv 1$? If so, then comparison methods of this paper will improve \eqref{AAB}
	to the following, which is the best possible bound of its type:
	As $t\to\infty$.
	\begin{equation}
		\frac{({\rm L}_\sigma^4+o(1))k(k^2-1)}{48\nu}
		\le \frac1t\log \E\left(|u_t(x)|^k\right) \le 
		\frac{({\rm Lip}_\sigma^4+o(1))k(k^2-1)}{48\nu}.
	\end{equation}
\end{OP}

\begin{OP} 
	For a more challenging problem consider (SHE) in the case that
	$\lim_{z\to\infty}\sigma(z)/z=\lambda>0$; this is an
	``asymptotic Anderson model.'' Is it true that the
	$\lim_{t\to\infty}t^{-1}\sup_{x\in\R}\log\E(|u_t(x)|^k)$
	exists for some integer $k\ge 2$? Theorem 2.10 of Foondun and Khoshnevisan
	\cite{FK} suggests that this might be the case provided that the initial
	profile is a bounded and measurable non-random function that satisfies
	$\inf_{x\in\R}u_0(x)\ge\eta$ for a sufficiently large $\eta>0$. However,
	a proof, or disproof, is completely missing.
\end{OP}

\section{Proof of Theorem \ref{th:main}}\label{sec:outline}
Let us say that $X^{(\varepsilon)}\approx Y^{(\varepsilon)}$ whenever
$X^{(\varepsilon)}:=\{X^{(\varepsilon)}_t(x)\}$ and 
$Y^{(\varepsilon)}:=\{Y^{(\varepsilon)}_t(x)\}$ 
are space-time random fields that are indexed as $t>0$ and 
$x\in A_\varepsilon\subset\R$,
and satisfy the following: 
\begin{equation}
	\sup_{t\in[0,T]}\sup_{x\in A_\varepsilon}
	\left\| X^{(\varepsilon)}_ t(x)-Y^{(\varepsilon)}_t(x)\right\|_k
	=o( \varepsilon^{\rho/2})
	\qquad(\varepsilon\downarrow 0),
\end{equation}
for all $T>0$, $k\ge 2$ and $0<\rho<\alpha-1$.

Our surrogates from $X^{(\varepsilon)}$ and $Y^{(\varepsilon)}$
will be selected from a list of 5 random fields. The first 4 are:
\begin{align}
	u^{(\varepsilon)}_t(x) &:=1+
		\int_{(0,t-\varepsilon^{\alpha})\times\R}p_{t-s}(y-x)
		\sigma\left(u_s\big(\varepsilon[y/\varepsilon]\big)\right) W(\d s\,\d y);\\
		\notag
	v^{(\varepsilon)}_t(x) &:=1+
		\int_{(0,t-\varepsilon^{\alpha})\times\R}p_{t-s}\left(
		\varepsilon[y/\varepsilon]-\varepsilon[x/\varepsilon]\right)
		\sigma\left(u_s\big(\varepsilon[y/\varepsilon]\big)\right) W(\d s\,\d y);\\\notag
	V_t^{(\varepsilon)}(x) &:= 1+
		\varepsilon^{-1}\int_{(0,t-\varepsilon^{\alpha})\times\R}
		P_{t-s}^{(\varepsilon)}\left(\varepsilon[y/\varepsilon]-\varepsilon[x/\varepsilon]\right)
		\sigma\left(u_s\big(\varepsilon[y/\varepsilon]\big)\right) W(\d s\,\d y);\\\notag
	W_t^{(\varepsilon)}(x)&:= 1+
		\varepsilon^{-1}\int_{(0,t)\times\R}P_{t-s}^{(\varepsilon)}
		\left(\varepsilon[y/\varepsilon]-\varepsilon[x/\varepsilon]\right)
		\sigma\left(u_s\big(\varepsilon[y/\varepsilon]\big)\right) W(\d s\,\d y),
\end{align}
where $[w]:=$ the greatest integer $\le w$ when $w\ge 0$, and
$[w]:=$ the smallest integer $\ge w$ when $w\le 0$.

There is a fifth random field $\bar{U}^{(\varepsilon)}:=
\{\bar{U}^{(\varepsilon)}_t(x)\}_{t\ge 0,x\in\varepsilon\Z}$ of interest to us; that 
is the solution to an infinite system of It\^o stochastic differential equations that is defined by
\begin{equation}\label{eq:ubar}
	\bar{U}^{(\varepsilon)}_t(x)=1+\varepsilon^{-1/2}
	\sum_{j=-\infty}^{\infty}\int_0^{t} P_{t-s}^{(\varepsilon)}
	\left(j\varepsilon -x\right)\sigma\left(\bar{U}_s^{(\varepsilon)}( j\varepsilon)\right)
	\d B_s^{(j,\varepsilon)}, 
\end{equation}
where
\begin{equation}
	B_t^{(j,\varepsilon)}:=\varepsilon^{-1/2}
	\int_{(0,t)\times (j\varepsilon,\,(j+1)\varepsilon)}W(\d s\,\d y). 
\end{equation}

Elementary properties of the Wiener integral
imply that $\lbrace B_t^{(j,\varepsilon)}\rbrace_{t\ge0,\, j\in \Z}$ is a 
[relabeled] version of the Brownian motions $\lbrace B_t(x)\rbrace_{t\ge0,\, x\in \Z}$
of the Introduction. As such, the theory of Shiga and Shimizu \cite{ShigaShimizu}
assures us that the infinite system \eqref{eq:ubar} of It\^o stochastic ODEs
has a unique strong solution. Moreover, 
the strong existence and uniqueness of the solutions to \eqref{eq:SHE'} 
and \eqref{eq:ubar} together imply that 
$\lbrace\bar{U}^{(\varepsilon)}_t(x)\rbrace_{t\ge0,\, x\in \varepsilon\Z}$ 
is a modification of $\lbrace U^{(\varepsilon)}_t(x)\rbrace_{t\ge0,\, x\in \varepsilon\Z}$. 
Thus, Theorem \ref{th:main} is a quantitative way to say that
$u\approx\bar{U}^{(\varepsilon)}$,
and follows once we verify the following chain of successive
approximations (with $A_{\varepsilon}=\varepsilon\Z$):
\begin{enumerate}
	\item[\bf (1).] $u\approx u^{(\varepsilon)}$;
	\item[\bf (2).] $u^{(\varepsilon)}\approx v^{(\varepsilon)}$;
	\item[\bf (3).] $v^{(\varepsilon)}\approx V^{(\varepsilon)}$;
	\item[\bf (4).] $V^{(\varepsilon)}\approx W^{(\varepsilon)}$; and
	\item[\bf (5).] $W^{(\varepsilon)}\approx \bar{U}^{(\varepsilon)}$.
\end{enumerate}

We now proceed to prove the theorem using the outlined five
steps. Henceforth,
all unimportant finite constants are denoted by
$C$. These constants might vary from line to line, and might depend only on 
the parameters $(T\,,\nu\,,\alpha)$.

\subsection{Verification of Assertion (1)} \label{sec:(1)}
We bound the mean square difference between $u_t(x)$ and 
$u_t^{(\varepsilon)}(x)$, using (BDG), as follows:
\begin{align} \label{eq:u-ue}\notag
	& \left\| u_t(x)-u_t^{(\varepsilon)}(x)\right\|_k^2\\
	& \le 2\cdot (4k)^{1/k} \int_{t-\varepsilon^{\alpha}}^t \d s
		\int_{-\infty}^\infty  \d y\ |p_{t-s}(y-x)|^2\left\|
		\sigma\left(u_s(y)\right)\right\|_k^2 \\
		\notag
	&\hspace{0.75cm}+ 2\cdot(4k)^{1/k}
		\int_{0}^{t-\varepsilon^{\alpha}} \d s \int_{-\infty}^\infty\d y\ 
		|p_{t-s}(y-x)|^2\left\|
		\sigma\left(u_s(y)\right)-\sigma\left(u_s\big(\varepsilon[y/\varepsilon]\big)\right)
		\right\|_k^2 \\\notag
	& \le C\int_0^{\varepsilon^{\alpha}}\d s \int_{-\infty}^\infty\d y\ |p_s(y)|^2
		+C\int_0^t \d s \int_{-\infty}^\infty\d y\ 
		|p_{t-s}(y)|^2\sup_{|z-z'|\le \varepsilon} \left\|
		u_s(z)-u_s(z') \right\|_k^2.
\end{align}
In the final step we have used the 
facts that: (i)  $\sigma$ is a Lipschitz continuous 
function; and (ii) 
$\sup_{0\le s\le T} \E(\vert u_s(y)\vert^k) <\infty$.
By the semigroup property and the inversion theorem of Fourier transforms,
\begin{equation} \label{eq:1:rmost:1}
	 \int_{-\infty}^\infty |p_s(y)|^2\,\d y
	 = (p_s*p_s)(0) = p_{2s}(0)=
	 \frac{1}{2\pi} \int_{-\infty}^\infty \e^{-2\nu \vert z\vert^{\alpha}s}
	 \,\d z=\frac{C}{s^{1/\alpha}}.
\end{equation}
Therefore, 
\begin{equation}\label{eq:1:rmost}
	\int_0^{\varepsilon^{\alpha}}\d s \int_{-\infty}^\infty\d y\ |p_s(y)|^2
	\le C\varepsilon^{\alpha-1}.
\end{equation}
This estimates the first term on the right-most side 
of \eqref{eq:u-ue}.

 Now we bound the second term. 
Thanks to the It\^o isometry \eqref{eq:qv} and the uniform bound 
of the $k$th moment of $u$, we conclude that whenever $|z-z'|\le\varepsilon$,
\begin{equation}
	\left\| u_s(z)-u_s(z')\right\|_k^2
	\le C\int_0^s\d r
	\int_{-\infty}^\infty\d w\left[p_{s-r}(w-z)-p_{s-r}(w-z')\right]^2.
\end{equation}
Therefore, Plancherel's identity yields the bound
\begin{align}\notag
	\left\|u_s(z)-u_s(z')\right\|_k^2 &\le C \int_0^s \d r
		\int_{-\infty}^\infty\d\zeta\ 
		\e^{-2\nu(s-r)\vert \zeta\vert^{\alpha}} \big\vert
		\e^{i(z-z')\zeta}-1\big\vert^2\\
		\notag
	&\hskip-.2in\le C\e^{2\nu s}\int_0^s \d r
		\int_0^\infty\d\zeta\ 
		\e^{-2\nu r(1+|\zeta|^{\alpha})} (1\wedge |z-z'|\zeta)^2\\
	&\hskip-.2in\le C\int_0^\infty\frac{(1\wedge 
		|z-z'|\zeta)^2}{1+\zeta^\alpha}\,\d \zeta.
\end{align}
We split the final integral according to whether or not
$\zeta<|z-z'|^{-1}$ and compute directly to find that
the preceding integral is at most $C |z-z'|^{\alpha-1}$.
Therefore, \eqref{eq:1:rmost:1} yields the inequality,
\begin{equation}\label{eq:2:rmost}
	\int_0^t\d s \int_{-\infty}^\infty\d y\ 
	|p_{t-s}(y)|^2\sup_{0\le z,z'\le \varepsilon} \left\|
	u_s(z)-u_s(z')\right\|_k^2 \le C\varepsilon^{\alpha-1}.
\end{equation}
Because $\alpha-1>\rho$,
Assertion (1) follows from \eqref{eq:u-ue},
\eqref{eq:1:rmost}, and \eqref{eq:2:rmost}.\qed

\subsection{Verification of Assertion (2)} \label{sec:(2)}
For every integer $k\ge 2$, the $k$th moment of $\sigma(u_s(y))$ is finite uniformly 
for all $y\in\R$ and $s\in[0\,,T]$. Therefore,
It\^o's isometry  yields
\begin{equation}
	\left\| u_t^{(\varepsilon)}(x)-v_t^{(\varepsilon)}(x)
	\right\|_k^2\le 
	C\int_{\varepsilon^{\alpha}}^t \d s \int_{-\infty}^\infty \d y\,
	\left[p_s(y)-p_s\left(\varepsilon[y/\varepsilon]\right)\right]^2
\end{equation}
It is well known that the stable probability density $p_1$ is $C^\infty$ 
and satisfies 
 \begin{equation}\label{eq:p'}
 	|p_1'(x)| \le \frac{C}{1+|x|^{\alpha+1}}\qquad\text{for all $x\in\R$}.
\end{equation}
When $1<\alpha<2$, this inequality is classical (see
G\l{}owacki \cite[Corollary (5.30)]{Glowacki} for some of the more recent
developments for stable densities over homogeneous spaces). 
When $\alpha=2$,  one has
$|p_1'(x)| \le C(1+|x|^3)^{-1}$ by explicit computation.
Consequently, scaling reveals that for all $\alpha\in(1\,,2]$, $s>0$,
and $x\in\R$,
\begin{equation}\label{eq:p':1}
	|p'_s(x)| =\left| s^{-2/\alpha}
	p_1'\left(\frac{x}{s^{1/\alpha}}\right)\right| \le \frac{Cs^{(\alpha-1)/\alpha}}{
	s^{(\alpha+1)/\alpha} +\vert x\vert^{\alpha+1}}.
\end{equation}
 With the above bound in hand, we obtain
\begin{equation}\begin{split}
	 &\left\| u_t^{(\varepsilon)}(x)-v_t^{(\varepsilon)}(x)\right\|_k^2\\
	&\le C\int_{\varepsilon^\alpha}^t\d s\int_0^\infty\d y\
		\left| \int_{(y-\varepsilon)_+}^y \frac{s^{(\alpha-1)/\alpha}}{
		s^{(\alpha+1)/\alpha} + x^{\alpha+1}}\, \d x\right|^2\\
	& \le C\varepsilon^2\int_{\varepsilon^\alpha}^t\d s\int_0^\infty\d y\
		\left| \frac{s^{(\alpha-1)/\alpha}}{
		s^{(\alpha+1)/\alpha} + (y-\varepsilon)_+^{\alpha+1}}\right|^2\\
	&=C\varepsilon^3\int_{\varepsilon^\alpha}^t \left| \frac{s^{(\alpha-1)/\alpha}}{
		s^{(\alpha+1)/\alpha}}\right|^2\,\d s +
		C\varepsilon^2\int_{\varepsilon^\alpha}^t\d s\int_0^\infty\d w\
		\left| \frac{s^{(\alpha-1)/\alpha}}{
		s^{(\alpha+1)/\alpha} + w^{\alpha+1}}\right|^2.
\end{split}\end{equation}
By scaling, the latter $\d w$-integral is equal to $C s^{-3/\alpha}$, whence
\begin{equation}\begin{split}
	 \left\| u_t^{(\varepsilon)}(x)-v_t^{(\varepsilon)}(x)\right\|_k^2
		&\le C\varepsilon^3\int_{\varepsilon^\alpha}^\infty \frac{\d s}{s^{4/\alpha}}+
		C\varepsilon^2\int_{\varepsilon^\alpha}^\infty\frac{\d s}{s^{3/\alpha}}\\
	&\le C\varepsilon^{\alpha-1}.
\end{split}\end{equation}
Because $\alpha-1>\rho$, this completes our demonstration of Assertion (2).\qed

\subsection{Verification of Assertion (3)} \label{sec:(3)}
The proof of Assertion (3) requires a certain form of a
local limit theorem for continuous-time random walks that
are in the domain of attraction of a symmetric $\alpha$-stable
L\'evy processes. As a first step in this direction we develop a
suitable quantitative form of a local central limit theorem for
continuous-time random walks that lie in the domain of attraction of
a symmetric $\alpha$-stable L\'evy process and satisfy the slightly stronger
Assumption \ref{ass}.

\begin{proposition}[A local CLT]\label{pr:locallimit}
	For all $T>0$ there exist positive and finite constants $K,C,\varepsilon_0$ such that 
	\begin{equation}\begin{split}
		&\sup_{x\in\varepsilon\Z} \left| 
			\frac{P_t^{(\varepsilon)}(x)}{\varepsilon}-  p_t(x) \right|\\
		&\hskip.8in\le C\times
			\begin{cases}\displaystyle
				\frac{\varepsilon^a
					|\ln \varepsilon|^{(a+\alpha)/\alpha}}{t^{(a+1)/\alpha}}
					&\text{if $t\ge K\varepsilon^\alpha|\ln\varepsilon|^{(a+\alpha)/a}$},\\\\
				t^{-1/\alpha}+\varepsilon^{-1}&\text{otherwise,}
			\end{cases}
	\end{split}\end{equation}
	uniformly for all $t\in(0\,,T]$ and $\varepsilon\in(0\,,\varepsilon_0)$.
\end{proposition}

When $\alpha=2$, the preceding is a sharp form of the
usual local central limit theorem for random walks.
The standard method to obtain such a result in continuous time
is to appeal to the classical local limit theorem for discrete-time random walks,
and then ``Poissonize.'' See Theorem 2.5.6 of Lawler and Limic \cite{LL}.
Proposition \ref{pr:locallimit} is a much stronger  version of the resulting estimate.
As was mentioned earlier, Proposition \ref{pr:locallimit} tells us that
$\varepsilon^{-1}P^{(\varepsilon)}_t(x)\approx p_t(x)$ when $\varepsilon$ is
small. The more significant part of our estimate is that 
it states that this approximation is good uniformly for all times 
$t\ge \varepsilon^{\alpha+o(1)}$; it is not hard
to see from the forthcoming proof that this time bound is basically sharp.
The proof is similar to---though
not identical to---the standard proof of the classical local limit
theorem for discrete-time random walks, but we work directly in
continuous time, rather than ``Poissonize.''

\begin{proof}
	By the Fourier inversion theorem,
	\begin{equation}
		p_t(x) = \frac1\pi\int_0^\infty\cos(zx)\e^{-\nu tz^\alpha}\,\d z \le
		p_t(0) = \frac{C}{t^{1/\alpha}}.
	\end{equation}
	Since $P^{(\varepsilon)}_t(x)=
	\P\{ X_{t/\varepsilon^{\alpha}}=\varepsilon^{-1}x\}\le 1$, it follows readily from
	the nearly-tautological inequality $|\varepsilon^{-1}P^{(\varepsilon)}_t(x)-p_t(x)|\le
	\varepsilon^{-1}P^{(\varepsilon)}_t(x)+p_t(x)$  that
	\begin{equation}
		\sup_{x\in\varepsilon\Z} \left| 
		\frac{P_t^{(\varepsilon)}(x)}{\varepsilon}-  p_t(x) \right|
		\le C\left( \varepsilon^{-1}+t^{-1/\alpha}\right).
	\end{equation}
	The preceding is valid for all $t,\varepsilon$, etc., but is useful to
	us in the case that $t<K\varepsilon^\alpha|\ln\varepsilon|^{(a+\alpha)/a}$. 
	From now on, we will concentrate on the more subtle case that
	$t\ge K\varepsilon^\alpha|\ln\varepsilon|^{(a+\alpha)/a}$, where the value of $K$ 
	will be determined during the course of the proof.
	
	Recall that 
	\begin{equation}
		\E\exp(iz X_{t/\varepsilon^{\alpha}})=
		\exp \left(-t\varepsilon^{-\alpha}[1-\hat{\mu}(z)] \right),
	\end{equation}
	for all $\varepsilon>0$, $t\ge 0$, and $z\in\R$.
	In accord with the inversion formula,
	\begin{equation}\label{FI}
		\frac{2\pi}{\varepsilon}P_t^{(\varepsilon)}(x)
		=\int_{-\pi/\varepsilon}^{\pi/\varepsilon} 
		\exp\left(-ixz -\frac{t[1-\hat{\mu}(\varepsilon z)]}{\varepsilon^{\alpha}}\right) \d z.
	\end{equation}
	Let us now choose and fix $\varepsilon\in(0\,,\nicefrac12)$. Assumption
	\ref{ass} and the uniform boundedness of $\hat\mu$ together
	guarantee that
	\begin{equation}\label{qqq}
		\left| 1 -\hat{\mu}(w) - \nu|w|^\alpha \right| \le C|w|^{a+\alpha}
		\qquad\text{for all $w\in\R$}.
	\end{equation}
	This bound has the following ready consequence: There exists $r_0\in(0\,,\pi)$
	such that
	\begin{equation}\label{r_0}
		1-\hat\mu(w) \ge \frac{\nu |w|^\alpha}{2}\qquad
		\text{whenever}\qquad |w|\le r_0.
	\end{equation}
	[It might help to recall that $\hat\mu$ is real-valued by symmetry.]
	
	Let us define
	\begin{equation}
		\lambda:= \left(\frac{(10+2a)\vert\ln\varepsilon\vert}{\nu t}\right)^{1/\alpha},
	\end{equation}
	for notational simplicity.  If $t>K\varepsilon^\alpha|\ln\varepsilon|^{(a+\alpha)/a}$,
	and $K>20+4a$, then
	\begin{equation}
		\lambda \le \left(\frac{|\ln\varepsilon|}{%
		2\varepsilon^\alpha|\ln\varepsilon|^{(a+\alpha)/a}}\right)^{1/\alpha}
		<\frac{r_0}{\varepsilon},
	\end{equation}
	for all $\varepsilon<\varepsilon_0$, as long as $\varepsilon_0>0$ is sufficiently small.
	In this case, we may split up the Fourier integral of \eqref{FI}
	into three regions as follows:
	\begin{equation}
		\frac{2\pi}{\varepsilon}P_t^{(\varepsilon)}(x) :=\mathbb{I}_1+\mathbb{I}_2+\mathbb{I}_3;
	\end{equation}
	where
	\begin{equation}
		\mathbb{I}_1:=\int_{|z| \le \lambda};\quad 
		\mathbb{I}_2:= \int_{\lambda<|z| \le r_0/\varepsilon};\quad 
		\mathbb{I}_3:=\int_{r_0/\varepsilon<|z| \le \pi/\varepsilon}.
	\end{equation}
	According to the inversion theorem,
	\begin{equation}
		2\pi p_t(x)=\int_{-\infty}^\infty \exp\left(-ixz
		-\nu t|z|^{\alpha}\right) \d z.
	\end{equation}
	Recall that our goal is to show that $\varepsilon^{-1} P^{(\varepsilon)}_t(x)\simeq
	p_t(x)$; with this aim in mind, we use a well-known 
	strategy. Namely, we will demonstrate that: 
	\begin{enumerate}
		\item[(i)] $\mathbb{I}_1\simeq 2\pi p_t(x)$; whereas 
		\item[(ii)] $\mathbb{I}_2\simeq\mathbb{I}_3\simeq 0$.
	\end{enumerate}
	
	Let us begin by estimating $\mathbb{I}_1$, whose analysis turns out
	to require the most care. Clearly,
	\begin{align}\notag
		&\sup_{x\in\varepsilon\Z} \left| \mathbb{I}_1-2\pi p_t(x)\right|\\
		&\le \int_{-\lambda}^\lambda \left|
			\exp\left(-\frac{t[1-\hat{\mu}(\varepsilon z)]}{\varepsilon^\alpha}
			\right)- \e^{-\nu t|z|^\alpha}\right|\d z
			+ 2\int_{\lambda}^{\infty} \e^{-\nu tz^{\alpha}}\,\d z\\\notag
		&= 2\int_0^\lambda \e^{-\nu tz^{\alpha}}\left|
			1-\exp\left(-\frac{t[1-\hat{\mu}(\varepsilon z)-\nu(\varepsilon z)^\alpha]}{\varepsilon^\alpha}
			\right) \right|\d z
			+ \frac{2}{t^{1/\alpha}}
			\int_{\lambda t^{1/\alpha}}^{\infty} \e^{-\nu y^{\alpha}}\,\d y.
	\end{align}
	Since $0<\lambda<r_0\varepsilon^{-1}<\pi\varepsilon^{-1}$, 
	Assumption \ref{ass} implies
	that uniformly for all $z\in(0\,,\lambda)$ and $t\in(0\,,T]$,
	\begin{equation}\begin{split}
		t\left|\frac{1-\hat{\mu}(\varepsilon z)-\nu(\varepsilon z)^\alpha}{\varepsilon^\alpha}
			\right| &\le Ct z^{a+\alpha}\varepsilon^a \\
		&< Ct\lambda^{a+\alpha}\varepsilon^a\\
		&= \frac{C\cdot(10+2a)^{(a+\alpha)/\alpha}}{t^{a/\alpha}}
			\varepsilon^a|\ln\varepsilon|^{(a+\alpha)/\alpha}.
	\end{split}\end{equation}
	Since $t>K\varepsilon^\alpha|\ln\varepsilon|^{(a+\alpha)/a}$
	for some $K>20+4a$, the preceding implies 
	that uniformly for all $z\in(0\,,\lambda)$ and $t\in(0\,,T]$,
	\begin{equation}
		t\left|\frac{1-\hat{\mu}(\varepsilon z)-\nu(\varepsilon z)^\alpha}{\varepsilon^\alpha}
		\right| \le \frac{C\cdot(10+2a)^{(a+\alpha)/\alpha}}{K^{a/\alpha}}
		<1,
	\end{equation}
	provided that we select $K$ sufficiently large. Henceforth, we choose
	and fix $K>20+4a$ large enough to ensure the preceding condition.
	 Since $|1-\exp(-w)|\le C|w|$ for all $w\in\mathbf{C}$ with $|w|<1$,
	it follows from Assumption \ref{ass} that whenever $t>K\varepsilon^\alpha
	|\ln\varepsilon|^{(a+\alpha)/a}$ for our choice of $K$,
	\begin{align}\notag
		&\sup_{x\in\varepsilon\Z} \left| \mathbb{I}_1-2\pi p_t(x)\right|\\
			\notag
		&\hskip.3in\le \frac{C\cdot (10+2a)^{(a+\alpha)/\alpha}}{t^{a/\alpha}}
			\varepsilon^a|\ln\varepsilon|^{(a+\alpha)/\alpha}
			\int_0^\lambda \e^{-\nu tz^{\alpha}}\,\d z
			+ \frac{2}{t^{1/\alpha}}
			\int_{\lambda t^{1/\alpha}}^{\infty} \e^{-\nu y^{\alpha}}\,\d y\\
		&\hskip.3in\le \frac{C\cdot (10+2a)^{(a+\alpha)/\alpha}}{t^{(a+1)/\alpha}}
			\varepsilon^a|\ln\varepsilon|^{(a+\alpha)/\alpha}
			+ \frac{2}{t^{1/\alpha}}
			\int_{\lambda t^{1/\alpha}}^{\infty} \e^{-\nu y^{\alpha}}\,\d y.
	\end{align}
	Because $\alpha>1$, l'H\^opital's rule implies that
	$\int_q^\infty \exp(-\nu y^\alpha)\,\d y \le  C\exp(-\nu q^\alpha/2)$ for all $q>0$. 
	Hence, for every  $t\in(K\varepsilon^\alpha|\ln\varepsilon|^{(a+\alpha)/a}\,,T]$,
	\begin{equation}\label{eq:I1}\begin{split}
		\sup_{x\in\varepsilon\Z} \left| \mathbb{I}_1-2\pi p_t(x)\right|
			&\le \frac{C\varepsilon^a|\ln\varepsilon|^{(a+\alpha)/\alpha}}{t^{(a+1)/\alpha}} + 
			\frac{C\e^{-\nu t\lambda^\alpha/2}}{t^{1/\alpha}}\\
		&=\frac{C\varepsilon^a|\ln\varepsilon|^{(a+\alpha)/\alpha}}{t^{(a+1)/\alpha}} + 
			\frac{C\varepsilon^{5+a}}{t^{1/\alpha}}\\
		&\le \frac{C\varepsilon^a|\ln\varepsilon|^{(a+\alpha)/\alpha}}{t^{(a+1)/\alpha}},
	\end{split}\end{equation}
	thanks to the definition of $\lambda$.
	
	Next we bound $\mathbb{I}_2$, using \eqref{r_0}, as follows:
	\begin{equation}\begin{split} \label{eq:I2}
		\sup_{x\in\varepsilon\Z} \vert \mathbb{I}_2\vert  &\le 
			\int_{\lambda<|z| \le r_0/\varepsilon} \exp\left(
			\frac{-t[1-\hat{\mu}(\varepsilon z)]}{\varepsilon^{\alpha}}\right) \d z\\
		&\le C\int_\lambda^\infty \e^{-t\nu z^\alpha/2}\,\d z\\
		&\le \frac{C\e^{-\nu t\lambda^\alpha/2}}{t^{1/\alpha}}\\
		&=\frac{C\varepsilon^{5+a}}{t^{1/\alpha}}.
	\end{split}\end{equation}
	
	Finally we bound the quantity $\mathbb{I}_3$ as follows: 
	Since $\hat\mu$ is continuous and real valued,
	thanks to symmetry, part (i) of Assumption \ref{ass} ensures that $\theta:=
	\sup_{r_0\le \vert z\vert \le \pi}\hat{\mu}(z)  <1$,
	and hence
	\begin{equation}\begin{split} \label{eq:I3}
		\sup_{x\in \varepsilon \Z} \vert \mathbb{I}_3 \vert 
			&\le\frac{C}{\varepsilon} \exp \left[-\frac{t(1-\theta)}{\varepsilon^{\alpha}}
			\right]\\
		&\le \frac{C}{\varepsilon} \exp \left(-\frac{t}{C\varepsilon^{\alpha}}\right)\\
		&\le \frac{C}{t^{1/\alpha}}\exp \left(-\frac{t}{C\varepsilon^{\alpha}}\right)\\
		&\le \frac{C\varepsilon^5}{t^{1/\alpha}},
	\end{split}\end{equation}
	uniformly for $t>K\varepsilon^\alpha|\ln\varepsilon|^{(a+\alpha)/a}$.
	Therefore, \eqref{eq:I1}, \eqref{eq:I2}, and \eqref{eq:I3} together imply
	the proposition.
\end{proof}

Proposition \ref{pr:locallimit} has the following consequence for us.

\begin{corollary}\label{co:locallimit}
	Choose and fix $T>0$. Then, 
	\begin{equation}
		\int_{\varepsilon^{\alpha}}^T\d t\int_{-\infty}^\infty\d y\ \left| 
		\frac{P_t^{(\varepsilon)}(\varepsilon[y/\varepsilon])}{\varepsilon}-  p_t(\varepsilon
		[y/\varepsilon]) \right|^2\le \varepsilon^{\alpha-1+o(1)}
		\quad\text{as $\varepsilon\downarrow 0$}.
	\end{equation}
	
\end{corollary}

\begin{proof}
	First of all, let us note that for all $\varepsilon\in(0\,,\nicefrac12)$
	and $t\in(0\,,T]$,
	\begin{equation}\label{eq:I1a}
		\int_{-\infty}^\infty \frac{P^{(\varepsilon)}_t(\varepsilon[y/\varepsilon])}{\varepsilon}
		\,\d y=\sum_{j=-\infty}^\infty P^{(\varepsilon)}_t(j\varepsilon)
		=1.
	\end{equation}
	The preceding has an analogue for the stable density $p_t$ as well. 
	In order to prove that, recall that we can realize our symmetric
	stable process as $\{B_{T(t)}\}_{t\ge 0}$, where $B$ denotes Brownian motion,
	and $T$ an independent stable-$2\alpha$ subordinator;
	see Bochner\cite[\S4.4]{Bochner}. 
	In particular,  it follows that we can write the stable heat kernel $p_t(x)$, via subordination,
	as follows:
	\begin{equation}
		p_t(x) = \int_0^\infty q_s(x) \tau_t(s)\,\d s,
		\quad\text{where}\quad
		q_s(x) := (2\pi s)^{-1/2}\exp\left( - \frac{x^2}{2s}\right)
	\end{equation}
	denotes the Brownian motion transition probability,
	and $\tau_t(\,\cdot)$ is a probability density on $(0\,,\infty)$
	for every $t>0$. Since $x\mapsto q_s(x)$ is nonincreasing on $(0\,,\infty)$, it follows that
	$p_t$ is, as well. We now use this conclusion as follows: 
	\begin{equation}\label{eq:I2a}\begin{split}
		\int_{-\infty}^\infty p_t(\varepsilon[y/\varepsilon])
			\,\d y &= 2\varepsilon\sum_{j=0}^\infty p_t(j\varepsilon)\\
		&=2\varepsilon p_t(0) + 2\sum_{j=1}^\infty \int_{(j-1)\varepsilon}^{j\varepsilon}
			p_t(j\varepsilon)\,\d y\\
		&\le 2\varepsilon p_t(0) + 2\int_0^\infty p_t(y)\,\d y\\
		&=\frac{C\varepsilon}{t^{1/\alpha}} + 1.
	\end{split}\end{equation}
	We may combine \eqref{eq:I1a} and \eqref{eq:I2a} as follows:
	\begin{equation}\begin{split}
		&\int_{-\infty}^\infty \left| \frac{P^{(\varepsilon)}_t
			(\varepsilon[y/\varepsilon])}{\varepsilon}
			-p_t(\varepsilon[y/\varepsilon])\right|^2\,\d y\\
		&\hskip0.25in\le \mathscr{M}_t(\varepsilon) 
			\cdot \int_{-\infty}^\infty \left[ 
			\frac{P^{(\varepsilon)}_t(\varepsilon[y/\varepsilon])}{\varepsilon}
			+ p_t(\varepsilon[y/\varepsilon])\right]\d y\\
		&\hskip0.25in \le \mathscr{M}_t(\varepsilon) 
			\cdot\left( \frac{C\varepsilon}{t^{1/\alpha}}
			+2\right),
	\end{split}\end{equation}
	where
	\begin{equation}
		\mathscr{M}_t(\varepsilon) := \sup_{y\in\varepsilon\Z}\left|
		\frac{P^{(\varepsilon)}(y)}{\varepsilon} -
		p_t(y)\right|. 
	\end{equation}
	In accord with Proposition \ref{pr:locallimit}, 
	\begin{align}\notag
		&\int_{\varepsilon^{\alpha}}^{K\varepsilon^\alpha|\ln\varepsilon|^{(a+\alpha)/a}}\d t
			\int_{-\infty}^\infty\d y\ \left| 
			\frac{P^{(\varepsilon)}_t(\varepsilon[y/\varepsilon])}{\varepsilon}
			-p_t(\varepsilon[y/\varepsilon])\right|^2\\\notag
		&\le C\int_{\varepsilon^{\alpha}}^{K\varepsilon^\alpha
			|\ln\varepsilon|^{(a+\alpha)/a}}
			\left(\frac{1}{t^{1/\alpha}}+\frac{1}{\varepsilon}\right)
			\left( \frac{\varepsilon}{t^{1/\alpha}} +2\right)\d t\\\notag
		&\le C\varepsilon\int_{\varepsilon^{\alpha}}^{%
			K\varepsilon^\alpha|\ln\varepsilon|^{(a+\alpha)/a}}
			\frac{\d t}{t^{2/\alpha}}+ C\int_{\varepsilon^{\alpha}}^{K\varepsilon^\alpha|
			\ln\varepsilon|^{(a+\alpha)/a}}
			\frac{\d t}{t^{1/\alpha}} 
			 + C\varepsilon^{\alpha-1}|\ln\varepsilon|^{(a+\alpha)/a}\\\notag
		&\le C \varepsilon^{\alpha-1}|\ln\varepsilon|^{(a+\alpha)/a}\\\notag
		&= \varepsilon^{\alpha-1+o(1)}.
	\end{align}
	Similarly,
	\begin{align}\notag
		&\int_{K\varepsilon^\alpha|\ln\varepsilon|^{(a+\alpha)/a}}^T\d t
			\int_{-\infty}^\infty\d y\ \left| \frac{P^{(\varepsilon)}_t(\varepsilon[y/\varepsilon])}{\varepsilon}
			-p_t(\varepsilon[y/\varepsilon])\right|^2\\\notag
		&\le C\int_{K\varepsilon^\alpha|\ln\varepsilon|^{(a+\alpha)/a}}^T
			\frac{\varepsilon^a
			|\ln \varepsilon|^{(a+\alpha)/\alpha}}{t^{(a+1)/\alpha}}
			\left( \frac{\varepsilon}{t^{1/\alpha}} +2\right)\d t\\
		&\le C\varepsilon^{a+1}|\ln \varepsilon|^{(a+\alpha)/\alpha}
			\int_{K\varepsilon^\alpha|\ln\varepsilon|^{(a+\alpha)/a}}^T
			\frac{\d t}{t^{(a+2)/\alpha}} \\\notag
		&\hskip1.5in +  C\varepsilon^a|\ln \varepsilon|^{(a+\alpha)/\alpha}
			\int_{K\varepsilon^\alpha|\ln\varepsilon|^{(a+\alpha)/a}}^T
			\frac{\d t}{t^{(a+1)/\alpha}}\\\notag
		&\le  \varepsilon^{\alpha-1+o(1)},
	\end{align}
	as $\varepsilon\downarrow0$, because $\alpha\in(1\,,2]$. The corollary follows
	from these computations.
\end{proof}

We are now ready to close this subsection by verifying Assertion (3).
The proof of Corollary \ref{co:locallimit} can be readily adjusted in order
to establish the following variation of Corollary \ref{co:locallimit}:
Uniformly for all $x\in\varepsilon\Z$ and $t\in(\varepsilon^{\alpha}\,,T]$,
\begin{equation}\begin{split}
	&\int_{\varepsilon^{\alpha}}^t\d s\int_{-\infty}^\infty\d y\ \left| 
		\frac{P_s^{(\varepsilon)} \left( \varepsilon[y/\varepsilon]-
		\varepsilon[x/\varepsilon]\right)}{\varepsilon}-  p_s\left(\varepsilon
		[y/\varepsilon]-\varepsilon[x/\varepsilon]\right) \right|^2\\
	&\hskip3.3in\le \varepsilon^{\alpha-1+o(1)},
\end{split}\end{equation}
as $\varepsilon\downarrow0$. We omit the minor remaining details. Since
$\E(|\sigma(u_t(z))|^k)$ is bounded uniformly in $z\in\R$
and $t\in(0\,,T]$ by a finite constant $K_{T,k}$, an application of the
Burkholder--Davis--Gundy inequality (BDG)  yields
\begin{align}\notag
	&\left\| v^{(\varepsilon)}_t(x) - V^{(\varepsilon)}_t(x)\right\|_k^2\\\notag
	&\le CK_{T,k}\int_{\varepsilon^{\alpha}}^t\d s\int_{-\infty}^\infty\d y\ \left| 
		\frac{P_s^{(\varepsilon)} \left( \varepsilon[y/\varepsilon]-
		\varepsilon[x/\varepsilon]\right)}{\varepsilon}-  p_s\left(\varepsilon
		[y/\varepsilon]-\varepsilon[x/\varepsilon]\right) \right|^2\\
	&\le \varepsilon^{\alpha-1+o(1)}.
\end{align}
Because $\alpha-1>\rho$, this verifies Assertion (3).\qed

\subsection{Verification of Assertion (4)} \label{sec:(4)}
Since $\E(|\sigma(u_s(y))|^k)$ is bounded uniformly in $y\in\R$ and
$s\in(0\,,T]$ for every integer $k\ge 2$, the Burkholder--Davis--Gundy 
inequality (BDG)  implies that, uniformly for all $\varepsilon\in(0\,,\nicefrac12)$,
$x\in\varepsilon\Z$, and $t\in(0\,,T]$,
\begin{equation}\begin{split}
	\left\| V_t^{(\varepsilon)}(x) -W_t^{(\varepsilon)}(x)\right\|_k^2
		&\le \frac{C}{\varepsilon^2}\int_{0}^{\varepsilon^{\alpha}} \d s 
		\int_{-\infty}^\infty\d y\ \left|
		P_s^{(\varepsilon)}\left(\varepsilon[y/\varepsilon] \right)\right|^2\\
	 & = \frac{C}{\varepsilon}\int_{0}^{\varepsilon^{\alpha}} 
	 	\sum_{j=-\infty}^\infty \left| P_s^{(\varepsilon)}(j\varepsilon)
		\right|^2 \,\d s.
\end{split}\end{equation}
Thanks to the Cauchy--Schwarz inequality,
\begin{equation}\label{eq:CS:P}
	\sum_{j=-\infty}^\infty \left| P_s^{(\varepsilon)}(j\varepsilon)\right|^2
	\le  \left| \sum_{j=-\infty}^\infty P_s^{(\varepsilon)}(j\varepsilon)\right|^2=1.
\end{equation}
Because $\alpha-1>\rho$, this proves Assertion (4). \qed

\subsection{Verification of Assertion (5)}
We begin with a technical lemma.
\begin{lemma} \label{lem:pbyeps}The following holds
\begin{equation}
	\sup_{0<\varepsilon<1}\,\sup_{\substack{0\le t\le T}}
	\int_0^t \d s \int_{-\infty}^\infty\d y\left[ \frac{P_{t-s}^{(\varepsilon)}\left(
	\varepsilon \left[y/\varepsilon\right] -x \right) }{\varepsilon}\right]^2 <\infty.
\end{equation}
\end{lemma}

\begin{proof}
	We first simplify the expression 
	\begin{align}\notag
		\int_0^t\d s \int_{-\infty}^\infty\d y
			\left[ \frac{ P_{t-s}^{(\varepsilon)}\left(\varepsilon 
			\left[y/\varepsilon\right] -x \right)}{\varepsilon} \right]^2 
			&=\frac{1}{\varepsilon}\int_0^t \d s \sum_{j=-\infty}^{\infty} \left[ 
			P_s^{(\varepsilon)}(j\varepsilon)\right]^2 \\\notag
		&=\varepsilon^{\alpha-1}\int_0^{t/\varepsilon^{\alpha}} \d r 
			\sum_{j=-\infty}^\infty |\P\{X_r=j\}|^2 \\
		&= \varepsilon^{\alpha-1}\int_0^{t/\varepsilon^{\alpha}} \d r \,\P\{Z_r=0\},
	\end{align}
	where $Z_r:=X_r-\tilde{X}_r$  for an independent copy $\tilde{X}$ of $X$.
	Of course, $Z$ is a random walk with the same jump distribution as $X$ 
	but twice the jump rate. Thus,
	\begin{equation}\begin{split}
		&\int_0^t \d s \int_{-\infty}^\infty\d y\left[ 
			\frac{P_{t-s}^{(\varepsilon)}\left(\varepsilon
			\left[y/\varepsilon\right] -x \right) }{\varepsilon}\right]^2\\
		&= \frac{\varepsilon^{\alpha-1}}{2}\int_0^{2t/\varepsilon^{\alpha}} 
			\P\{ X_r=0\}\,\d r\\
		&=\frac{\varepsilon^{\alpha-1}}{4\pi} \int_{-\pi}^{\pi} 
			\left\lbrace1-\exp\left(\frac{-2t[1-\hat{\mu}(\xi)]}{%
			\varepsilon^{\alpha}}\right)\right\rbrace
			\frac{\d\xi}{1-\hat{\mu}(\xi)},\label{finalINT}
	\end{split}\end{equation}
	where we used the Fourier inversion formula for the last line. 
	Due to condition (i) of Assumption \ref{ass}, we can find constants 
	$0<C_1<C_2$ such that $C_1\vert \xi\vert^{\alpha} <
	1-\hat{\mu}(\xi)<C_2\vert \xi \vert^{\alpha}$ for $\vert \xi\vert \le \pi$. 
	With this observation in mind,
	we may split the final integral of \eqref{finalINT}
	into two integrals: One is computed over the region  $\{\vert \xi\vert \le \varepsilon\}$,
	the other over $\{\vert \xi\vert > \varepsilon\}$. 
	For the first integral, we use the bound $1-\e^{-x}\le x$ 
	valid for $x\ge0$ and for the second integral, we use the bound $1-\e^{-x}\le 1$.
	In this way we find that
	\begin{equation}\begin{split}
		&\int_0^t \d s \int_{-\infty}^\infty\d y\left[ 
			\frac{P_{t-s}^{(\varepsilon)}\left(\varepsilon 
			\left[y/\varepsilon\right] -x \right) }{\varepsilon}\right]^2\\
		&\hskip1in\le \frac{\varepsilon^{\alpha-1}}{4\pi}\int_{\vert \xi\vert \le \varepsilon}
			\frac{2t}{\varepsilon^{\alpha}} \,\d\xi+  
			\frac{\varepsilon^{\alpha-1}}{4\pi}
			\int_{\varepsilon<\vert \xi\vert \le \pi}\frac{1}{C_1\vert\xi\vert^{\alpha}}\d \xi.
	\end{split}\end{equation}
	This is a uniformly bounded quantity.
\end{proof}

We begin our derivation of Assertion (5) with the observation that for $x\in \varepsilon \Z$,
\begin{align}
	&\bar{U}_t^{(\varepsilon)}(x)- W_t^{(\varepsilon)}(x) \\\notag
	&=\frac{1}{\varepsilon} \int_{(0,t)\times \R}P_{t-s}^{(\varepsilon)}
		\left(\varepsilon \left[y/\varepsilon\right] -x \right) \left[ 
		\sigma\left(\bar{U}_s\left(\varepsilon \left[y/\varepsilon\right]\right)\right)
		-\sigma\left(u_s\left(\varepsilon \left[y/\varepsilon\right]\right)\right)
		\right] W(\d s\,\d y).
\end{align}
Consequently, we may apply the Burkholder--Davis--Gundy inequality
(BDG), together with the triangle inequality, in order to find that
the nonrandom function
\begin{equation}
	\mathscr{D}(t) := \sup_{x\in\varepsilon\Z}
	\left\|\bar{U}_t^{(\varepsilon)}(x)- W_t^{(\varepsilon)}(x)\right\|_k^2 
	\qquad(t\ge 0)
\end{equation}
satisfies the recursive inequality
\begin{equation}
	\mathscr{D}(t) 
	\le  \frac{C}{\varepsilon^2} \int_0^t
	\left( \mathscr{D}(s) + \mathscr{E}(s)\right)
	\d s\int_{-\infty}^\infty\d y\
	\left[P_{t-s}^{(\varepsilon)} \left(\varepsilon \left[y/\varepsilon\right] -x \right)\right]^2,
\end{equation}
where
\begin{equation}
	\mathscr{E}(s) := \sup_{y\in\varepsilon\Z}
	\left\| u_s\left(\varepsilon\left[y/\varepsilon\right]\right)- 
	W_s^{(\varepsilon)}\left(\varepsilon\left[y/\varepsilon\right]\right)\right\|_k^2
	\qquad(s\ge 0).
\end{equation}
The already-proven Assertions (1)--(4) together show that
$\sup_{s\in(0,T]}\mathscr{E}(s)\le C\varepsilon^{\alpha-1}$
for all sufficiently small $\varepsilon$.
Consequently, we arrive at the following recursive inequality, valid
uniformly for every $t\in(0\,,T]$ and $\varepsilon>0$ sufficiently
small:
\begin{align}
	\mathscr{D}(t) 
		&\le  C\varepsilon^{\alpha-1}\int_0^t\d s\int_{-\infty}^\infty\d y\
		\left| \frac{P_{t-s}^{(\varepsilon)}\left(\varepsilon \left[y/\varepsilon\right] -x \right)}{
		\varepsilon} \right|^2 \\\notag
	&\hskip1.5in + C \int_0^t\d s\int_{-\infty}^\infty\d y\
		\left| \frac{P_{t-s}^{(\varepsilon)}\left(\varepsilon \left[y/\varepsilon\right] -x \right)}{
		\varepsilon} \right|^2 \mathscr{D}(s)\\\notag
	&\le C\varepsilon^{\alpha-1}+ C \int_0^t\d s\,\mathscr{D}(s)\int_{-\infty}^\infty\d y\
		\left| \frac{P_{t-s}^{(\varepsilon)}\left(\varepsilon \left[y/\varepsilon\right] -x \right)}{
		\varepsilon} \right|^2 .
\end{align}
Gronwall's inequality reveals that $\mathscr{D}(t) \le C \varepsilon^{\alpha-1} $. 
In order to use that inequality, we  need to know that
$\sup_{0\le t\le T} \mathscr{D}(t)<\infty$; this follows from 
the known bounds
$\sup_{0\le t\le T} \E(\vert u_t(x)\vert^k)<\infty$ and
$\sup_{0\le t\le T} \E(\vert \bar{U}^{(\varepsilon)}_t(x)\vert^k)<\infty$;
see \cite{GJKMS}. 
\qed

\section{Proof of the remaining results}
For the proof of Theorem \ref{th:SA}, we will need to estimate the 
H\"older continuity of the random function
$t\mapsto \bar{U}_t^{(\varepsilon)}$. The following address this H\"older
continuity matter.

\begin{lemma} \label{lem:ubar:hold} 
	For every real $T>0$ and integer $k\ge 2$,
	\begin{equation}
		\sup_{x\in\varepsilon\Z}
		\E\left(\left|\bar{U}_t^{(\varepsilon)}(x)-
		\bar{U}_s^{(\varepsilon)}(x)\right|^k\right) \le C 
		\left(\frac{t-s}{\varepsilon}\right)^{k/2},
	\end{equation}
	uniformly for all $0<s<t\le T$.
\end{lemma}

\begin{proof}
	Since $\sup_{ 0<\varepsilon<1} \sup_{x\in \varepsilon\Z} \sup_{t\le T}
	\E(|\bar{U}_t^{(\varepsilon)}(x)|^k)<\infty$, (BDG) implies that
	\begin{align}
		&\left\|\bar{U}_t^{(\varepsilon)}(x)-\bar{U}_s^{(\varepsilon)}(x)\right\|_k^2 \\
			\notag
		&\le \frac{C}{\varepsilon} \int_s^t 
			\sum_{j=-\infty}^\infty \left| P_{t-r}^{(\varepsilon)}(j\varepsilon -x)\right|^2\d r 
			+ \frac{C}{\varepsilon}\int_0^s \sum_{j=-\infty}^\infty
			\left| P_{t-r}^{(\varepsilon)}(j\varepsilon -x)-P_{s-r}^{(\varepsilon)}(j\varepsilon -x)
			\right|^2\d r.
	\end{align}
	The first term is bounded by $ C(t-s)/\varepsilon$ thanks to \eqref{eq:CS:P}. We appeal to 
	Plancherel's identity---using also the fact that $-1\le \hat\mu(z)\le 1$
	for all $z$---in order to obtain the following 
	equivalent representation of the corresponding second term:
	\begin{align}\notag
		&\frac{C}{\varepsilon} \int_0^s \d r \int_{-\pi}^{\pi}\d z \ 
			\left| \exp\left(-\frac{(t-r)[1-\hat{\mu}(z)]}{\varepsilon^{\alpha}}\right)-
			\exp\left(-\frac{(s-r)[1-\hat{\mu}(z)]}{\varepsilon^{\alpha}}\right)
			\right|^2 \\\notag
		&=\frac{C}{\varepsilon} \int_0^s \d r \int_{-\pi}^{\pi}\d z \ 
			\exp\left(-\frac{2(s-r)[1-\hat{\mu}(z)]}{\varepsilon^{\alpha}}\right)
			\left|1- \exp\left(-\frac{(t-s)[1-\hat{\mu}(z)]}{\varepsilon^{\alpha}}\right)
			\right|^2 \\
		&\le\frac{C}{\varepsilon} \int_{-\pi}^{\pi}
			\frac{\varepsilon^{\alpha}}{1-\hat{\mu}(z)}\cdot
			\left|1- \exp\left(-\frac{(t-s)[1-\hat{\mu}(z)]}{\varepsilon^{\alpha}}\right)
			\right|^2 \,\d z\\\notag
		&\le \frac{C(t-s)}{\varepsilon},
	\end{align}
	where we used the inequality $1-\e^{-x}\le \sqrt{x}$,
	valid for all $x\ge 0$.
\end{proof}

\begin{proof}[Proof of Theorem \ref{th:SA}] 
	The proof is a ready consequence of Theorem \ref{th:main} 
	and a standard application
	of the Kolmogorov continuity criterion \cite[Theorem 4.3, page 10]{minicourse},
	whose details follow.
	
	Choose and fix some integer $k\in[2\,,\infty)$.
	According to \cite[p.\ 566]{FK}, 
	\begin{equation}
		\E \left(| u_t(x)-u_s(x)|^k \right)\le C\vert t-s\vert^{\eta k},
	\end{equation}
	uniformly for all $x\in\R$ and $0\le s, t \le T$, where 
	\begin{equation}
		\eta:= \frac{\alpha-1}{2\alpha}.
	\end{equation}	
	Therefore, we appeal to the Kolmogorov continuity theorem and deduce the following
	bound: For every integer $k\ge 4/(3\eta)$,
	\begin{equation}
		\sup_{x\in\R}
		\E\left(\sup_{\substack{0\le s,t \le T \\ | t-s|<\varepsilon^3T}}
		| u_t(x)-u_s(x) |^k\right) \le C \varepsilon^{3\eta k-4},
	\end{equation}
	simultaneously for every $\varepsilon\in(0\,,1)$.
	Lemma \ref{lem:ubar:hold} and a second appeal to
	Kolmogorov's continuity theorem, together yield the following:
	\begin{equation}
		\sup_{x\in\varepsilon\Z}
		\E\left( \sup_{\substack{0\le s,t \le T \\ \vert t-s\vert<\varepsilon^3T}}\left| 
		\bar{U}^{(\varepsilon)}_t(x)-\bar{U}^{(\varepsilon)}_s(x)\right|^k\right)
		\le C \varepsilon^{k-4} \le C \varepsilon^{3\eta k-4},
	\end{equation}
	simultaneously for every $\varepsilon\in(0\,,1)$. The preceding
	two observations can be combined as follows. Define
	\begin{equation}
		D^{(\varepsilon)}_t(x) := u_t(x) - \bar{U}_t^{(\varepsilon)}(x).
	\end{equation}
	Then
	\begin{equation}\label{eq:DD}
		\sup_{x\in\varepsilon\Z}
		\E\left( \sup_{\substack{0\le s,t \le T \\ \vert t-s\vert<\varepsilon^3T}}\left| 
		D^{(\varepsilon)}_t(x)-D^{(\varepsilon)}_s(x)\right|^k\right)
		\le C \varepsilon^{3\eta k-4},
	\end{equation}
	simultaneously for every $\varepsilon\in(0\,,1)$.
	
	Next, we notice that
	\begin{equation}
		(0\,,T] \subseteq \bigcup_{1\le j\le \varepsilon^{-3}} \left( j\varepsilon^3T
		\,,(j+1)\varepsilon^3T\right],
	\end{equation}
	whence, for every fixed $\nu\in(0\,,\alpha-1)$,
	\begin{align}\notag
		&\P\left\{ \left| u_t(x) -\bar{U}_t^{(\varepsilon)}(x) \right| > 
			\varepsilon^{\nu/2} \text{ for some $x\in \varepsilon\Z, 
			|x|\le \varepsilon^{-M},\text{ and }t\in (0\,,T]$}\right\} \\
			\notag
		&\le 2\varepsilon^{-M-1}\sup_{x\in\varepsilon\Z}
			\P\left\{ \sup_{t\in(0,T]}\left| u_t(x) -\bar{U}_t^{(\varepsilon)}(x) \right| 
			> \varepsilon^{\nu/2} \right\}\\
		&\le J_1+J_2,
	\end{align}
	where
	\begin{equation}\begin{split}
		J_1&:= 2\varepsilon^{-M-4}\sup_{x\in\varepsilon\Z}\sup_{t\in(0,T]}
			\P\left\{ \left| u_t(x) -\bar{U}_t^{(\varepsilon)}(x) \right| 
			> \tfrac12\varepsilon^{\nu/2} \right\},\\
		J_2&:=2\varepsilon^{-M-1}\sup_{x\in\varepsilon\Z}
			\P\left\{\sup_{\substack{0<s,t\le T\\|t-s|
			<\varepsilon^3T}}\left| D_t^{(\varepsilon)}(x)
			-D^{(\varepsilon)}_s(x) \right| > \tfrac12\varepsilon^{\nu/2}\right\}.
	\end{split}\end{equation}
	Choose and fix $\rho\in(\nu\,,\alpha-1)$ and appeal to Theorem \ref{th:main},
	and the Chebyshev inequality, in order to see that
	\begin{equation}
		J_1 \le C \varepsilon^{-M-4+k(\rho-\nu)/2}\qquad\text{as 
		$\varepsilon\downarrow 0$},
	\end{equation}
	where $k\ge 2$ is an integer that we can choose [and fix] as large as we wish.
	And \eqref{eq:DD} ensures that
	\begin{equation}
		J_2 \le C \varepsilon^{-M-5+k(6\eta-\nu)/2}\qquad\text{as 
		$\varepsilon\downarrow 0$},
	\end{equation}
	for [say] the same integer $k$. Since $6\eta>\alpha-1>\nu$
	and $k$ can be chosen to be as large as
	needed, we have proved the following: For every $\nu\in(0\,,\alpha-1)$
	and $Q>0$,
	\begin{align}\notag
		&\P\left\{ \left| u_t(x) -\bar{U}_t^{(\varepsilon)}(x) \right| > 
			\varepsilon^{\nu/2} \text{ for some $x\in \varepsilon\Z, 
			|x|\le \varepsilon^{-M},\text{ and }t\in (0\,,T]$}\right\} \\
			\notag
		&\hskip3.5in=o( \varepsilon^Q)\qquad\text{as $\varepsilon\downarrow 0$}.
	\end{align}
	This is more than enough to imply the result.
\end{proof}

\begin{proof}[Proof of Theorem \ref{th:comp}] 
	Choose and fix an arbitrary positive constant $\beta$.
	Let $u$ be the solution to (SHE), and define
	\begin{equation}
		v_t(x) := u_{\beta t}(x)\qquad(t\ge 0\,, x\in\R).
	\end{equation}
	Then, $v$ solves the stochastic heat equation,
	\begin{equation}
		\frac{\partial}{\partial t} v_t(x) = 
		\beta(\mathcal{L}v_t)(x) +\beta^{1/2} \sigma(v_t(x))\tilde\xi(t\,,x),
	\end{equation}
	subject to $v_0(x)=u_0(x)$ for all $x$, where $\tilde\xi = \tilde{\xi}^{(\beta)}$ is a
	space-time white noise. This observation shows us that 
	Theorem \ref{th:comp} holds for a \emph{fixed} value of $\nu>0$
	if and only if Theorem \ref{th:comp} is valid for \emph{every} possible choice
	of $\nu>0$. We have seen (Examples \ref{ex:alpha=2}
	and \ref{ex:alpha<2}) that there exists $\nu>0$, together with
	a random walk model that satisfies Assumption \ref{ass} for that
	particular value of $\nu$. [In fact, every $\nu>0$ works
	when $\alpha=2$; see Example \ref{ex:alpha=2}.]
	We choose and fix that $\nu$ and random walk from now on.
	
	The remainder of the proof of the Theorem follows from 
	applying a result of Cox, Fleischmann, and Greven \cite[Theorem 1]{CFG}
	along with our Theorem \ref{th:main}. 
	Theorem 1 of \cite{CFG} is a comparison of moments for interacting diffusions. 
	For every fixed $\varepsilon,t>0$ and
	$x_1,x_2,\cdots,x_m \in \R$, that comparison theorem implies that
	\begin{equation}\label{eq:UU}
		\E\left(\prod_{j=1}^m U_t^{(\varepsilon)}\left(\varepsilon [x_j/\varepsilon]\right)\right)
		\le 
		\E\left(\prod_{j=1}^m U_t^{1,(\varepsilon)}\left(\varepsilon [x_j/\varepsilon]\right)\right)
	\end{equation}
	where $U_t^{1,(\varepsilon)}$ solves \eqref{mild_d} with 
	$\sigma$ replaced by $\bar{\sigma}$. Our assumption that
	$\sigma(0)=\bar{\sigma}(0)=0$ guarantees that both $U_t^{(\varepsilon)}$ 
	and $U_t^{1,(\varepsilon)}$ remain nonnegative;
	this follows from the well-known comparison principle
	for interacting diffusions; see for example, 
	Gei\ss\ and Manthey \cite[Theorem 1.1]{GeissManthey},
	Mueller \cite[Theorem 3.1]{Mueller}, or Georgiou et al
	\cite{GJKMS}. The support of $U_t^{(\varepsilon)}(x)$ and 
	$U_t^{1,(\varepsilon)}(x)$ is thus $\R_+$ and so the conditions of 
	Cox et al \cite[Theorem 1]{CFG} are verified. Now let
	$\varepsilon\downarrow 0$ in \eqref{eq:UU} and appeal to our Theorem \ref{th:main}
	to conclude the proof.
\end{proof}

\begin{proof}[Proof of Theorem \ref{th:lyapunov}]
	The upper bound of ${\rm Lip}_\sigma^4k^3/\nu$ follows
	immediately from Example 2.9 of Foondun and Khoshnevisan
	\cite{FK}. We will establish the lower bound next.
	Let $\theta\in\R$ be fixed,
	and consider  the solution $X_t(x)$ to the \emph{parabolic Anderson model},
	\begin{equation}\label{PAM}
		\frac{\partial}{\partial t} X_t(x) = \nu\frac{\partial^2}{\partial x^2}
		X_t(x) + {\rm L}_\sigma X_t(x)\xi(t\,,x),
	\end{equation}
	subject to $X_0(x)\equiv 1$ for all $x$. That is, 
	$(t\,,x)\mapsto X_t(x)$ solves  (SHE)
	with $\sigma(x)={\rm L}_\sigma x$.
	According  to Theorem \ref{th:comp},
	\begin{equation}
		\E\left( |u_t(x)|^k\right) \ge \E\left(\left| X_t(x)\right|^k\right)
		\ge\exp\left(\frac{{\rm L}_\sigma^4 k(k^2-1)t}{48\nu}\right),
	\end{equation}
	pointwise.
	See Theorem of 6.4 of Khoshnevisan \cite[Theorem 6.4]{Kh:CBMS}
	for the last bound. This proves the theorem.
\end{proof}

\noindent\textbf{Acknowledgements.}
	We thank Nicos Georgiou for many enjoyable conversations
	on this topic and Gennady Samorodnitsky for his helpful comments
	on the topic of Example \ref{ex:alpha<2}.
	Many thanks are also due to Ivan Corwin and Roger Tribe for their question about
	Lyapunov estimates; that question is answered here in Theorem \ref{th:lyapunov}.

\begin{small}\bigskip
\noindent\textbf{Mathew Joseph} (\texttt{m.joseph@sheffield.ac.uk}).\\
\noindent Department of Probability and Statistics, University of Sheffield,
	Sheffield, England, S3 7RH, UK\\
 
\noindent\textbf{Davar Khoshnevisan} (\texttt{davar@math.utah.edu}).\\
\noindent Department of Mathematics, University of Utah,
		Salt Lake City, UT 84112-0090, USA\\ 
		
\noindent\textbf{Carl Mueller} (\texttt{carl.2013@outlook.com}).\\
\noindent Department of Mathematics, University of Rochester,
	Rochester, NY 14627, USA\\
\end{small}

\end{document}